\documentclass[article]{amsart}
\usepackage{amsmath,amsthm,amssymb}
\usepackage{hyperref}
\usepackage{cite}

\usepackage{graphicx}
\usepackage{tikz}
\usetikzlibrary{calc}
\usetikzlibrary{positioning}
\usepackage{subfigure}

\usepackage[a4paper,
            left=30mm,
            right=30mm,
            top=30mm,
            bottom=30mm]{geometry}

\newcommand{\Var} {\mathrm {Var}}
\newcommand{\Prob} {\mathbb {P}}
\newcommand{\prob}[1]{\Prob\left(#1\right)}

\newcommand{\nc}{\mathcal{N}}
\newcommand{\tc}{\mathcal{T}}
\newcommand{\htc}{\hat{\tc}}
\newcommand{\st}{\textbf{t}}
\newcommand{\N}{N_{\st}}

\newcommand{\n}{n_{\st}}

\newcommand{\hf}{\tilde{f}_{p,q}}
\newcommand{\hF}{\tilde{F}_{p,q}}

\newcommand{\E}{\mathbb{E}}

\newcommand{\hxi}{\hat{\xi}}

\newcommand\dto{\overset{\mathrm{d}}{\to}}

\newcommand{\giventhat}{\@ifstar\@giventhatstar\@giventhatnostar}

\title{Asymptotic normality for general subtree counts in  conditioned Galton--Watson trees}
\author{Fameno Rakotoniaina }
\thanks{Fameno Rakotoniaina acknowledges support from the DAAD In-Region Programme (SUN Mathematics, South Africa, 2021; grant no.57552766), as well as additional support from the Faculty of Science and the Mathematics Division at Stellenbosch University.}

\author{Dimbinaina Ralaivaosaona}
\thanks{Dimbinaina Ralaivaosaona acknowledges support from the National Research Foundation of South Africa (NRF), Award No.~RA231127167905.}
\address{Department of Mathematical Sciences \\ Stellenbosch University \\ Private Bag X1 \\ Matieland 7602, South Africa.}
\email{frakotoniaina@aimsric.org / naina@sun.ac.za}

\newtheorem{theorem}{Theorem}
\newtheorem{lemma}[theorem]{Lemma}
\newtheorem{corollary}[theorem]{Corollary}
\newtheorem{proposition}[theorem]{Proposition}

\newtheorem{assumption}{Assumption}
\theoremstyle{definition}

\newtheorem*{conjecture*}{Conjecture}
\newtheorem*{proposition*}{Proposition}
\theoremstyle{remark}

\newtheorem*{remark*}{Remark}
\newtheorem{example}[theorem]{Example}


\begin{document}
\begin{abstract}

Let $\mathcal{T}$ denote a Galton--Watson tree with offspring distribution $\xi$ satisfying $\mathbb{E}(\xi)  = 1$, and let $\mathcal{T}_n$ be the Galton--Watson tree conditioned to have exactly $n$ nodes. We show that, under a mild moment condition on $\xi$, the number of occurrences of a fixed rooted plane tree $\mathbf{t}$ as a general subtree in $\mathcal{T}_n$ is asymptotically normal as $n \to \infty$, with both mean and variance linear in $n$. In addition, we prove that this limiting distribution is nondegenerate except for some special cases where the variance remains bounded. These results confirm a conjecture of Janson in recent work on the same topic. Finally, we present examples showing that if the proposed moment condition on $\xi$ is violated, the conclusion may fail.

\end{abstract}

\maketitle

\section{Introduction}
  
Throughout this paper, all trees are rooted and ordered (or plane), i.e., the order of the root branches matters. Let $\st$ be a fixed tree, and for any other tree $T$, let $\N(T)$ denote the number of occurrences of $\st$ as a general subtree of $T$, and let $\n(T)$ denote the number of occurrences of $\st$ attached to the root of $T$. We use the same notation as in Janson~\cite{J21}, where these parameters are rigorously defined; see \cite[Section~2]{J21}.

Informally, a rooted tree $T$ is drawn upside down, with the root as the highest node. An occurrence of a rooted tree $\st$ as a general subtree of $T$ is an embedding of $\st$ into $T$ (as rooted ordered trees) such that the image of the root of $\st$ is the highest node in its image in $T$, and for each node $v$ of $\st$, the degree of its image in $T$ is greater than or equal to the degree of $v$ in $\st$.

The tree functional $\N$ satisfies the so-called \emph{additive property}, with $\n$ as its associated \emph{toll function}. That is, if a tree $T$ has root branches $T_1, T_2, \dots, T_d$, then
\[
\N(T)=\sum_{k=1}^{d} \N(T_k) + \n(T).
\]

Subtree pattern counts are natural parameters to study in random trees, and they have been investigated in various forms throughout the literature. The most common and extensively studied pattern is the so-called \emph{fringe-subtree pattern}. A fringe subtree of a tree $T$ is the subtree induced by a node together with all its descendants. Results on fringe-subtree counts can be found, for example, in~\cite{MR1102319,J16} and the references therein. Another type of subtree pattern—defined via tree embeddings in which the degrees of the internal nodes are preserved—was studied in~\cite{CHYZAK_DRMOTA_KLAUSNER_KOK_2008}.
   
We only consider the Galton–Watson tree $\tc$ with a critical offspring distribution $\xi$ satisfying the following standard conditions (which we assume to hold from now on without further mention):
\begin{itemize}
    \item $\mathbb{P}(\xi=0)>1$, 
    \item $\E(\xi)=1$, and 
    \item $\gcd\{k \ :\ \mathbb{P}(\xi=k)>0\}=1$.
\end{itemize}
The first condition ensures that the Galton--Watson process can die out; the second is the criticality condition; and the third is required only to guarantee that $\mathbb{P}(|\tc| = n) > 0$ for all sufficiently large $n$ (where, $|T|$ denotes the number of nodes in tree $T$). If the third condition is omitted and, for example, $\gcd\{k : \mathbb{P}(\xi = k) > 0\} = d > 1$, then one must restrict to $n \equiv 1 \pmod d$. We shall always assume that $\Var(\xi)<\infty$, although this will follow from the assumptions imposed throughout the paper.

Under this setting, Janson~\cite{J21} proved a law of large numbers for $\N(\tc_n)$, where $\tc_n$ is the Galton--Watson tree conditioned on having $n$ nodes, with offspring distribution $\xi$ satisfying $\E(\xi^{\Delta}) < \infty$, and where $\Delta := \Delta(\st)$ denotes the maximum outdegree of the nodes in $\st$. In the same paper (see \cite[Section~5]{J21}), it was conjectured that, under an additional suitable moment condition on $\xi$, the random variable $\N(\tc_n)$ should satisfy a nondegenerate central limit theorem, except for a few exceptional cases. This has already been verified in the case where $\xi$ is bounded; see \cite[Proposition~5.1]{J21}. Here, we extend this result further.
\begin{theorem}\label{main theorem}
Let $\tc_n$ be the conditioned Galton--Watson tree with offspring distribution $\xi$, satisfying $\E(\xi^{2\Delta(\st)+1}) < \infty$. Then, as $n \to \infty$, the mean and variance of $\N(\tc_n)$ satisfy the asymptotic estimates
\[
\mathbb{E}(\N(\tc_n)) = \mu n + o(\sqrt{n})
\quad \text{and} \quad
\Var(\N(\tc_n)) = \gamma^2 n + o(n),
\]
where $\mu = \mathbb{E}(\n(\tc))$, and $\gamma \ge 0$ is a constant. Moreover, if
\[
\limsup_n \Var(\N(\tc_n)) = \infty,
\]
then $\gamma \neq 0$, and
\[
\frac{\N(\tc_n) - \mu n}{\gamma \sqrt{n}} \dto \nc(0,1).
\]
\end{theorem}
An explicit expression for $\mu$, in terms of the degree sequence of $\st$, can be found in \cite[Equation~(3.2)]{J21}. However, our method does not provide an explicit expression for the constant $\gamma$.

The proof of this theorem uses a truncation argument together with variance estimates. This approach was suggested by Janson in \cite{J21}, and it has already been applied to prove a general central limit theorem in \cite[Theorem~1]{ralaivaosaona2020}. Although the toll function $\n(\cdot)$ is local for a fixed $\st$, the result in \cite[Theorem~1]{ralaivaosaona2020} does not apply in our setting because $\n(\cdot)$ fails to satisfy the condition stated in Equation~(2) of that theorem when $\st$ has height at least $2$. Therefore, a different approach is required.

We remark that the condition $\E(\xi^{2\Delta+1}) < \infty$ is already strong enough to cover most classical families of random trees, including rooted labelled trees and plane trees; see \cite{D09} for further examples. In Section~\ref{sec5}, we show that this moment assumption is not far from optimal in the sense that if $\E(\xi^{2\Delta}) = \infty$, then there exists a tree $\st$ with maximum outdegree $\Delta$ for which the variance $\Var(\N(\tc_n))$ grows superlinearly in $n$.

Throughout, asymptotic statements are taken as $n \to \infty$ unless explicitly stated otherwise. We write $f(n)=O(g(n))$ to mean that there exist constants $C>0$ and $n_0$ such that $|f(n)| \le C\,|g(n)|$ for all $n \ge n_0$. We write $f(n)=o(g(n))$ if $f(n)/g(n) \to 0$, and $f(n) \sim g(n)$ if $f(n)/g(n) \to 1$. We also use Vinogradov notation: $f(n) \ll g(n)$ means $f(n)=O(g(n))$, and $f(n) \gg g(n)$ means $g(n) \ll f(n)$. 
\section{Preliminaries}

In this section, we establish several auxiliary lemmas; the first is the most important. We also prove the mean estimate in Theorem~\ref{main theorem}. Before stating our key lemma, let us first fix notation. For a tree $T$ and a positive integer $M$, the \emph{cut-off tree} $T^{(M)}$ is the subtree induced by the nodes at distance at most $M$ from the root. The number of nodes at distance exactly $M$ from the root in $T$ is denoted by $w_M(T)$.

The infinite size-biased tree, also known as the Kesten tree~\cite{K86}, is denoted by $\htc$. Its distribution can be defined locally by 
\[
\forall M,\  \forall T,\  \prob{\htc^{(M)}=T}=w_M(T)\prob{\tc^{(M)}=T}.
\]
In particular, its root-degree distribution, denoted by $\hxi$, is defined by $\mathbb{P}(\hxi = k) = k\,\mathbb{P}(\xi = k)$. For a comprehensive background on Galton--Watson trees and many useful results, see \cite{J06,J12}. Throughout, we follow the notation of \cite{ralaivaosaona2020}.
\begin{lemma}\label{expectation of products of n(T) with finite tree size}
Let $\st$ and $\st'$ be fixed trees with maximum outdegree at most $\Delta$ and height at most $M$. Then the following statements hold:
\begin{itemize}
\item[(a)] If $\mathbb{E}(\xi^{\Delta+2}) < \infty$, then $\mathbb{E}(\n(\tc)\,|\tc^{(M)}|) < \infty$ and $\mathbb{E}(\n(\htc)\,|\htc^{(M)}|) < \infty$.\label{st:a}
\item[(b)] If $\mathbb{E}(\xi^{\Delta+2}) < \infty$, then $\mathbb{E}(\n(\tc_n)\,|\tc_n^{(M)}|) = O(1)$ as $n \to \infty$.
\item[(c)] If $\mathbb{E}(\xi^{2\Delta+1}) < \infty$, then $\mathbb{E}(n_{\st}(\tc_n)\,n_{\st'}(\tc_n)) = O(1)$ as $n \to \infty$.
\end{itemize}
\end{lemma}
\begin{proof}
We prove these statements by induction on the height of $\st$.

For the first statement of Part (a), assume $M=1$, i.e., $\st$ is a star; suppose that $\st$ has $d\leq \Delta$ leaves. Then
\[
\n(\tc)=\binom{\deg(\tc^{(M)})}{d}
\quad\text{and}\quad
|\tc^{(M)}|=1+\deg(\tc^{(M)}),
\]
where $\deg(\cdot)$ denotes the root degree. Since $M=1$, the root degree $\deg(\tc^{(M)})$ has the same distribution as $\xi$. Taking expectations yields
\[
\mathbb{E}\bigl(\n(\tc)\,|\tc^{(M)}|\bigr)
= \E\!\left((1+\xi)\binom{\xi}{d}\right)
\le \E\!\left((1+\xi)^{\Delta+1}\right),
\]
which is finite because $\E(\xi^{\Delta+2})<\infty$ implies $\E(\xi^{\Delta+1})<\infty$.

Similarly,
\[
\mathbb{E}\bigl(\n(\htc)\,|\htc^{(M)}|\bigr)
= \E\!\left((1+\hxi)\binom{\hxi}d\right)
\le \E\!\left((1+\hxi)^{\Delta+1}\right).
\]
This expectation is finite provided $\E(\hxi^{\Delta+1})<\infty$. Since $\hxi$ is the size-biased version of $\xi$ (and $\E\xi=1$), we have
\[
\E(\hxi^{\Delta+1})=\E(\xi^{\Delta+2})<\infty.
\]
This establishes the base case $M=1$.

Now assume that $M > 1$, and let $t_1, t_2, \dots, t_d$ be the branches of $\st$, ordered from left to right. Clearly, the maximum outdegrees and heights of these branches are bounded above by $\Delta$ and $M-1$, respectively. Condition on the event $\deg(\tc) = k$, and let $T_1, T_2, \dots, T_k$ be the branches of $\tc$, also ordered from left to right. Then
\[
\n(\tc)\,|\tc^{(M)}|
= \left(1 + \sum_{i=1}^{k} \bigl|T_i^{(M-1)}\bigr|\right)
  \sum_{(i_1,\dots,i_d)} \prod_{j=1}^{d} n_{t_j}(T_{i_j}),
\]
where the sum on the right-hand side is over all increasing $d$-tuples $(i_1,\dots,i_d)$ with
$1 \le i_1 < \cdots < i_d \le k$. Taking the conditional expectation yields
\begin{equation}\label{eq:exp-T}
\mathbb{E}\!\left( \n(\tc)\,|\tc^{(M)}| \,\middle|\, \deg(\tc)=k \right)
= \mathbb{E}\!\left( \sum_{(i_1,\dots,i_d)} \prod_{j=1}^{d} n_{t_j}(T_{i_j}) \right)
  + \mathbb{E}\!\left( \sum_{\substack{1\le i\le k \\ (i_1,\dots,i_d)}} 
     \bigl|T_i^{(M-1)}\bigr| \prod_{j=1}^{d} n_{t_j}(T_{i_j}) \right).
\end{equation}
Since the $T_i$ are independent copies of $\tc$, the first term equals
\[
\sum_{(i_1,\dots,i_d)} \prod_{j=1}^{d} \mathbb{E}\bigl(n_{t_j}(\tc)\bigr).
\]
By the induction hypothesis, $\mathbb{E} \left( n_{t_j}(\tc) \right)$ is finite for every $j$. So the above term is also finite. For the second term on the right-hand side of \eqref{eq:exp-T}, two cases can occur:
\begin{itemize}
    \item If $T_i = T_{i_j}$ for some $j$, say $j_0$, then
    \[
    \mathbb{E} \left( \sum_{\substack{(i_1, i_2, \dots, i_d)}} \left| T_i^{(M-1)} \right| \prod_{j=1}^{d} n_{t_j}(T_{i_j}) \right)
    = \sum_{\substack{(i_1, i_2, \dots, i_d)}} \mathbb{E} \left( \left| T_i^{(M-1)} \right| n_{t_{j_0}}(T_i^{(M-1)}) \right) \prod_{j \neq j_0} \mathbb{E} \left( n_{t_j}(T_{i_j}) \right).
    \]
    Since the $T_j$'s are independent copies of $\tc$, the above is equal to
    \[
    \sum_{\substack{(i_1, i_2, \dots, i_d)}} \mathbb{E} \left( \left| \tc^{(M-1)} \right| n_{t_{j_0}}(\tc) \right) \prod_{j \neq j_0} \mathbb{E} \left( n_{t_j}(\tc) \right).
    \]
    By the induction hypothesis, each factor is finite and independent of $k$, so this term is finite.

    \item If $T_i \neq T_{i_j}$ for any $j$, then the $T_{i_j}$'s, together with $T_i$, are independent of each other, and we have
    \[
    \mathbb{E} \left( \sum_{\substack{(i_1, i_2, \dots, i_d)}} \left| T_i^{(M-1)} \right| \prod_{j=1}^{d} n_{t_j}(T_{i_j}) \right)
    = \sum_{\substack{(i_1, i_2, \dots, i_d)}} \mathbb{E} \left( \left| T_i^{(M-1)} \right| \right) \prod_{j=1}^{d} \mathbb{E} \left( n_{t_j}(T_{i_j}) \right),
    \]
    which is equal to
    \[
    \sum_{\substack{(i_1, i_2, \dots, i_d)}} \mathbb{E} \left( \left| \tc^{(M-1)} \right| \right) \prod_{j=1}^{d} \mathbb{E} \left( n_{t_j}(\tc) \right).
    \]
    Once again, by the induction hypothesis, this is finite (note that $\mathbb{E} ( | T_i^{(M-1)} | )$ is finite by \cite[Lemma~3]{J16}, and the other factors are finite by the induction hypothesis).
\end{itemize}
Combining the two cases, we have some constant terms summed over $i \in \{1, 2, \dots, k\}$ and $(i_1, \dots, i_d)$ with the above properties. Thus, there exists a constant $C \geq 0$, independent of $k$, such that
\[
\mathbb{E} \left( \n(\tc) |\tc^{(M)}| \mid \deg(\tc) = k \right) \leq C (k+1) \binom{k}{d} \leq C (k+1)^{\Delta+1}.
\]
Taking the expectation again, we get
\[
\mathbb{E} \left( \n(\tc) |\tc^{(M)}| \right) \leq C \mathbb{E} \left( (\xi+1)^{\Delta+1} \right),
\]
which we have already established to be finite. Hence, we are done for the first part of Statement~(a).

The second part of Statement (a) can be handled in the same way. Indeed, conditioning on the event $\deg(\htc) = k$, the tree $\htc$ consists of a root attached to a copy of $\htc$ and $k - 1$ copies of $\tc$, and all these branches are independent of one another. The only difference is that, in this case, we may encounter terms such as $\mathbb{E}( n_{t_j}(\htc) |\htc^{(M-1)}| )$, $\mathbb{E}( |\htc^{(M-1)}| )$, and $\mathbb{E}( n_{t_j}(\htc) )$. However, these terms are all finite by the induction hypothesis. Thus, we deduce that there exists a constant $C' \geq 0$ such that
\[
\mathbb{E}(\n(\htc) |\htc^{(M)}|) \leq C' \mathbb{E} \left( (\hxi + 1)^{\Delta + 1} \right),
\]
which is finite, as we have seen before. Therefore, the proof of Part (a) is complete.

For Part (b), we also proceed by induction in a similar fashion. So, for $M = 1$, (assuming that $\st$ has $d\leq \Delta$ leaves) we have  
$\n (\tc_n) = \binom{\deg(\tc_n)}{d}$ and $|\tc_n^{(M)}| = 1 + \deg(\tc_n)$. Hence,
\[
\mathbb{E}(\n(\tc_n) |\tc_n^{(M)}|) \leq \mathbb{E}\left((1 + \deg(\tc_n))^{\Delta + 1}\right) = O(\mathbb{E} \, w_1(\tc_n)^{\Delta + 1}).
\]
Applying \cite[Lemma 2.2]{J06}, we deduce that the above is a bounded term.

Now assume that $M > 1$, and condition on the event that $\deg(\tc_n) = k$ and that the branches $T_1, T_2, \dots, T_k$ (ordered from left to right) of $\tc_n$ have sizes $n_1, n_2, \dots, n_k$, respectively. Then, these branches of $\tc_n$ are $k$ independent copies of $\tc_{n_1}, \tc_{n_2}, \dots, \tc_{n_k}$. Using the same argument as in the proof above, we have
\[
\mathbb{E}(n_{t_j}(\tc_n) |\tc_n^{(M-1)}|) = O(1),
\]
for any branch $t_j$ of $\st$, by the induction hypothesis. Hence, we deduce that
\[
\mathbb{E} \left( \n(\tc_n) |\tc_n^{(M)}| \,\Big|\, \deg(\tc_n) = k \;\wedge\; (|T_1| = n_1, |T_2| = n_2, \dots, |T_k| = n_k) \right) = O((1 + k)^{\Delta + 1}).
\]
Taking the expectation again, we obtain
\[
\mathbb{E}(\n(\tc_n) |\tc_n^{(M)}|) = O \left( \mathbb{E}(1 + \deg(\tc_n))^{\Delta + 1} \right) = O(1).
\]
The proof of Part (b) is complete.

For Part (c), we also proceed by induction. If $M = 1$, i.e., $\st$ and $\st'$ are both stars (with $d$ and $d'$ leaves respectively), then we have
\begin{align*}
	\mathbb{E}(n_{\st}(\tc_n) n_{\st'}(\tc_n)) 
	& = \mathbb{E} \left( \binom{\deg(\tc_n)}{d} \binom{\deg(\tc_n)}{d'} \right) \\[.5em]
	& \leq \mathbb{E} \left( \deg(\tc_n)^{2\Delta} \right) \\
	& = \mathbb{E} \left( w_1(\tc_n)^{2\Delta} \right),
\end{align*}
which is bounded above by a constant, see \cite[Lemma 2.2]{J06}, under the assumption that $\mathbb{E}(\xi^{2\Delta + 1}) < \infty$.
For $M > 1$, let $t_1,\, t_2, \dots,\, t_d$ be the branches of $\st$ and $t'_1,\, t'_2,\, \dots,\, t'_q$ the branches of $\st'$. As before, we condition on the event that $\deg(\tc_n) = k$, with $T_1,\, T_2,\, \dots,\, T_k$ denoting the branches of $\tc_n$, and let $n_i = |T_i|$ for $1 \leq i \leq k$. Then, we have
\[
n_{\st}(\tc_n) = \sum_{(i_1,i_2,\dots,i_d)} \prod_{j=1}^{d} n_{t_j}(T_{i_j})
\]
\[
n_{\st'}(\tc_n) = \sum_{(i'_1,i'_2,\dots,i'_q)} \prod_{j=1}^{q} n_{t'_j}(T_{i'_j})
\]
where the summations are over $(i_1,i_2,\dots,i_d)$ and $(i'_1,i'_2,\dots,i'_q)$ (arranged in ascending order), defined similarly as before. Thus,
\begin{align*}
&\mathbb{E}\left(n_{\st}(\tc_n)n_{\st'}(\tc_n)\,\middle|\, \deg(\tc_n)=k \wedge (|T_1|=n_1, |T_2|=n_2, \dots, |T_k|=n_k)\right) \\
&= \sum_{\substack{(i_1,i_2,\dots,i_d)\\(i'_1,i'_2,\dots,i'_q)}} \mathbb{E} \left( \prod_{j=1}^{d} n_{t_j}(T_{i_j}) \prod_{\ell=1}^{q} n_{t'_\ell}(T_{i'_\ell}) \right)
\end{align*}
Since the branches are independent, the expectations that appear in the product above are of the form $\mathbb{E} \, n_{t_j}(\tc_{n_{i_j}})$, $\mathbb{E} \, n_{t'_\ell}(\tc_{n_{i'_\ell}})$, or $\mathbb{E} \left(n_{t_j}(\tc_{n_{i_j}}) n_{t'_\ell}(\tc_{n_{i_j}})\right)$, each of which is bounded (by the induction hypothesis). Hence,
\[
\mathbb{E}\left(n_{\st}(\tc_n)n_{\st'}(\tc_n)\,\middle|\, \deg(\tc_n)=k \wedge (|T_1|=n_1, |T_2|=n_2, \dots, |T_k|=n_k)\right) = O\left(\binom{k}{d}\binom{k}{q}\right)
\]
Therefore,
\begin{align*}
\mathbb{E}\left(n_{\st}(\tc_n)n_{\st'}(\tc_n)\right)
&= O\left(\mathbb{E} \left(\binom{\deg(\tc_n)}{d}\binom{\deg(\tc_n)}{q}\right)\right) \\[.5em]
&= O\left(\mathbb{E} \, \deg(\tc_n)^{2\Delta}\right),
\end{align*}
which is already known to be bounded under our assumption, so the proof of the lemma is complete.
\end{proof}

We now present the first application of Lemma~\ref{expectation of products of n(T) with finite tree size}, which is a version \cite[Lemma 3.5]{J21} with an explicit bound on the error term but under a slightly stronger condition on $\xi$.

\begin{lemma}\label{evaluate n(Tn)}
	Assume that $\xi$ satisfies $\mathbb{E}(\xi^{\Delta+2})<\infty$. Then, we have
	\begin{equation}\label{eq:Mean-Last}
	|\E \n(\tc_n) - \E \n(\htc)| = O(n^{-1/2}).
	\end{equation} 
\end{lemma}

\begin{proof}
First, by definition, we have
\begin{align*}
    \left|\E \n(\tc_n) - \E \n(\htc)\right| 
    &= \left|\sum_{T} \n(T) \prob{\tc_n^{(M)} = T} - \sum_{T} \n(T) \prob{\hat{\tc}^{(M)} = T} \right| \\
    &\leq S_1 + S_2,
\end{align*}
where 
\begin{align*}
    S_1 &:= \sum_{|T| \leq n/2} \n(T) \left|\prob{\tc_n^{(M)} = T} - \prob{\hat{\tc}^{(M)} = T}\right|, \\
    S_2 &:= \sum_{|T| > n/2} \n(T) \left|\prob{\tc_n^{(M)} = T} - \prob{\hat{\tc}^{(M)} = T}\right|.
\end{align*}

To estimate $S_1$, we make use of Equation~(5.42) in the proof of \cite[Lemma~5.9]{J16}. For $|T| \leq n/2$, the result states that
\begin{equation}\label{eq:ntohat}
    \prob{\tc_n^{(M)} = T} - \prob{\hat{\tc}^{(M)} = T} = \prob{\hat{\tc}^{(M)} = T} \cdot O\left(\frac{|T|}{n^{1/2}}\right).
\end{equation}
Hence,
\begin{alignat*}{2}
    S_1 
    &= \sum_{|T| \leq n/2} \n(T) \prob{\hat{\tc}^{(M)} = T} \cdot O\left(\frac{|T|}{n^{1/2}}\right)\\ 
    &\ll n^{-1/2} \sum_T \n(T)|T| \prob{\hat{\tc}^{(M)} = T}\\
    &\ll n^{-1/2} \E\left(\n(\hat{\tc})|\hat{\tc}^{(M)}|\right).
\end{alignat*}
The latter is bounded by $O(n^{-1/2})$ by Part~(a) of Lemma~\ref{expectation of products of n(T) with finite tree size}.

\medskip

Next, we estimate $S_2$. We trivially have 
\[
S_2 \leq \sum_{|T| > n/2} \n(T) \prob{\tc_n^{(M)} = T} + \sum_{|T| > n/2} \n(T) \prob{\hat{\tc}^{(M)} = T}.
\]
We rewrite the first summation on the right-hand side as follows:
\begin{alignat*}{2}
    \sum_{|T| > n/2} \n(T) \prob{\tc^{(M)} = T} 
    &= \sum_{k \geq 0} k \prob{|\tc^{(M)}| > \tfrac{n}{2} \wedge \n(T) = k} \\
    &= \sum_{k \geq 0} k \prob{\n(\tc^{(M)}) = k} \cdot \prob{|\tc^{(M)}| > \tfrac{n}{2} \mid \n(\tc^{(M)}) = k}.
\end{alignat*}
By Markov's inequality, we obtain
\[
\prob{|\tc^{(M)}| > \tfrac{n}{2} \mid \n(\tc^{(M)}) = k} \leq \frac{2}{n} \E\left(|\tc^{(M)}| \mid \n(\tc^{(M)}) = k\right).
\]
Therefore,
\begin{alignat*}{2}
    \sum_{|T| > n/2} \n(T) \prob{\tc^{(M)} = T} 
    &\leq \frac{2}{n} \sum_{k \geq 0} k \prob{\n(\tc^{(M)}) = k} \E\left(|\tc^{(M)}| \mid \n(\tc^{(M)}) = k\right) \\
    &= \frac{2}{n} \E\left(\n(\tc^{(M)}) \cdot \E\left(|\tc^{(M)}| \mid \n(\tc^{(M)})\right)\right) \\
    &= \frac{2}{n} \E\left(\n(\tc^{(M)}) |\tc^{(M)}|\right) \\
    &= O(n^{-1}),
\end{alignat*}
where we used Lemma~\ref{expectation of products of n(T) with finite tree size} in the final step.

Similarly, by the same argument, we obtain
\[
\sum_{|T| > n/2} \n(T) \prob{\hat{\tc}^{(M)} = T} \leq \frac{2}{n} \E\left(\n(\hat{\tc}^{(M)}) |\hat{\tc}^{(M)}|\right) = O(n^{-1}).
\]
Thus, we conclude that $S_1 + S_2 = O(n^{-1/2})$, completing the proof of the lemma.
\end{proof}

The next lemma gives an asymptotic estimate of the expectation $\mathbb{E}(\N(\tc_n))$.
\begin{lemma}\label{expectation of Nt}
	Assume that $\xi$ satisfies $\mathbb{E}(\xi^{\Delta+2})<\infty$. Then, we have
	\[\mathbb{E}(\N(\tc_n)) =\mu n +o(\sqrt{n})\]
 as $n\to\infty$, where $\mu=\E (\n(\tc))$.
\end{lemma}
\begin{proof}
Consider the centered toll function \( f(T) = \n(T) - \E(\n(\htc)) \). The associated additive functional is then given by
\[
F(T) = \N(T) - |T| \E(\n(\htc)).
\]
We trivially have
\[
\n(\htc) \leq \n(\htc)\,|\htc^{(M)}| \quad \text{and} \quad \n(\tc) \leq \n(\tc)\,|\tc^{(M)}|.
\]
On the other hand, by Lemma~\ref{expectation of products of n(T) with finite tree size}, both \( \E(\n(\tc)\,|\tc^{(M)}|) \) and \( \E(\n(\htc)\,|\htc^{(M)}|) \) are finite. Therefore, we conclude that \( \E|f(\tc)| \) is also finite.

\medskip
Moreover, by Lemma~\ref{evaluate n(Tn)}, we have
\begin{equation} \label{lkkljdk}
\E\left|f(\tc_n)\right| = \E\left|\n(\tc_n) - \E(\n(\htc))\right| = O(n^{-1/2}).
\end{equation}
In particular, the first part of \cite[Theorem~1.5]{J16} applies, and as a consequence we obtain
\[
\E\left(F(\tc_n)\right) = n\, \E f(\tc) + o(\sqrt{n}).
\]
Writing \( F(\tc_n) \) in terms of \( \N(\tc_n) \), we deduce that
\[
\E \N(\tc_n) - n\, \E(\n(\htc)) = n\, \E f(\tc) + o(\sqrt{n}),
\]
which is equivalent to the equation stated in the lemma.
\end{proof}

Next, we shift our attention to the variance. As we will see in the next section, the proof of the central limit theorem relies on a truncation lemma, which is mainly based on convergence in \(L^2\). Therefore, estimating the variance plays a crucial role in the rest of the proof. We will need to consider the truncated version of our toll function. To that end, we define for every $p,q\in \mathbb{N}\cup \{\infty\}$  the additive functional $F_{p,q}$ associated with the following toll function: 
\[
f_{p,q}(T)=
\begin{cases}
    \n(T)\ &\text{ if }\  p \leq |T|<q \\
    0\ &\text{otherwise}.
\end{cases}
\]
Observe that, in particular, we can write our original functional $\N$ as $F_{1,\infty}$. Furthermore, for each positive integer $k$, let 
\[
\mu_k:=\mathbb E \n(\tc_k).
\]
Then, we define $\hF$ to be the ``centered" additive functional with the toll function 
\[
\hf(T)=\begin{cases}
    \n(T)-\mu_{|T|}\ &\text{ if }\  p \leq |T|<q \\
    0\ &\text{otherwise}.
\end{cases}
\]
We have the following lemma.
\begin{lemma}\label{lem:toll_add}
	Assume that $\xi$ satisfies $\mathbb{E}(\xi^{2\Delta+1})<\infty$ and let $\hF$ and $\hf$ be as defined above. Then, we have
	\begin{equation}
	\E\big(\hf(\tc_n)\hF(\tc_n)\big)=O(1),
	\end{equation}
 as $n\to\infty$, where the implied constant is independent of $p$ and $q$.
\end{lemma}

\begin{proof}
We can assume that \( p \leq n < q \); otherwise, the expression inside the expectation equals zero, and there is nothing to prove.  By definition
\[
\hF(\tc_n)=\sum_{v\in\tc_n}\hf(\tc_{n,v}),
\]
where $\tc_{n,v}$ denotes the fringe subtree of $\tc_n$ rooted at $v$. The strategy is to split the sum according to the distance $d(v)$ from $v$ to the root. So, we define
	\[
		\Sigma_1 := \sum_{d(v)<M}\hf(\tc_{n,v}), \ \text{and }\  \Sigma_2 := \sum_{d(v)\geq M}\hf(\tc_{n,v}).
	\]
 
Let us look at the contribution from $\Sigma_1$ first. Observe that
\[
\E\left(\hf(\tc_n)\Sigma_1\right) = \sum_{\ell=0}^{M-1}\mathbb{E}\left(\hf(\tc_n)\sum_{d(v)=\ell}\hf(\tc_{n,v})\right).
\]
For each $\ell$, let $t_1, \ t_2, \cdots, \ t_d $ be the branches of $\st$ rooted at nodes at level $\ell$. Conditioning on $\tc_n^{(\ell)}$, let $T_1,\ T_2, \cdots ,T_k$ be the branches $\tc_n$ rooted at the nodes at level $\ell$. Further, we also condition on the sizes of $T_1,\ T_2, \cdots T_k$, say $n_1,\ n_2, \cdots , n_k$. Thus, we have
\[
\mathbb{E}\left(\hf(\tc_n)\sum_{d(v)=\ell}\hf(\tc_{n,v})\Big|\tc_n^{(\ell)}=T \wedge (\forall i \  (|T_i|=n_i)) \right)=
\]
\begin{equation}\label{expectation of centered f, proof}
    \E\left(\left(\sum_{\pi}\prod_{j=1}^{d}n_{t_j}(T_{\pi_j})-\mu_n\right)\sum_{\substack{1\leq i\leq k\\ p\leq n_i <q}}\left(\n(T_i)-\mu_{n_i}\right)\right),
\end{equation}
where the first summation inside the expectation is over possible embeddings $\pi$ of $\st^{(\ell)}$ into $T^{(\ell)}$, and for each embedding $\pi$, the node ranked $j$-th at level $\ell$ in $\st^{(\ell)}$ is mapped to the node ranked $\pi_j$ in $T^{(\ell)}$.

We know from Lemma \ref{expectation of products of n(T) with finite tree size} that $(\mu_k)_{k\geq 1}$ is a bounded sequence.  Expanding the term inside the expectation \eqref{expectation of centered f, proof} and using the fact the $T_i$'s are independent copies of conditioned Galton--Watson trees, the expectations of the non-constant terms are of the form $\E(\n(\tc_{n_i}))$,  $\E(n_{t_j}(\tc_{n_i}))$ or $\E(n_{t_j}(\tc_{n_i})\n(\tc_{n_i}))$; these are all bounded terms according to Lemma \ref{expectation of products of n(T) with finite tree size}. Hence
\[
   \mathbb{E}\left(\hf(\tc_n)\sum_{d(v)=\ell}\hf(\tc_{n,v})\Big|\tc_n^{(\ell)}=T \wedge (\forall i \  (|T_i|=n_i))\right)=O\left(w_{\ell}(T)n_{\st^{(\ell)}}(T^{(\ell)})\right)
\]
Taking the expectation again and summing over $\ell\in \{0,1,2,\cdots, M-1\}$, (noting that $w_{(\ell)}(T)\leq |T^{(\ell)}|$) we deduce from Lemma \ref{expectation of products of n(T) with finite tree size} that 
\[
\E\left(\hf(\tc_n)\Sigma_1\right)=O\left(\sum_{\ell=0}^{M-1}\E\left(n_{\st^{(\ell)}}(\tc_n)|\tc_n^{(\ell)}|\right)\right)=O(1).
\]

We take the same approach for the contribution from $\Sigma_2$, and we have  
\[
\E\left(\hf(\tc_n)\Sigma_2\right) = \sum_{\ell\geq M}\mathbb{E}\left(\hf(\tc_n)\sum_{d(v)=\ell}\hf(\tc_{n,v})\right).
\]
However, we know that the height of $\st$ is $M$. Hence, for $\ell\geq M$, $\hf(\tc_n)$ is completely determined by $\tc_n^{(\ell)}$. Thus,
\[
\mathbb{E}\left(\hf(\tc_n)\sum_{d(v)=\ell}\hf(\tc_{n,v})\Big|\tc_n^{(\ell)}=T \wedge (\forall i \  (|T_i|=n_i) )\right)=\hf(T)\sum_{i=1}^{k}\E(\hf(\tc_{n_i}))=0.
\]
Therefore, we deduce using the law of total expectation that $\E\left(\hf(\tc_n)\Sigma_2\right)=0$. Putting all contributions together proves the lemma. 
\end{proof}
We are now ready to prove the linear bound on the variance. Let us state the result.
\begin{lemma}\label{lem:varesti}
Assume that $\xi$ satisfies $\mathbb{E}(\xi^{2\Delta+1})<\infty$ and let $F_{p,q}$ be as defined previously. Then, there exists a function $\varepsilon(p)$ that is independent of $q$ and $n$, such that 
\[
\Var(F_{p,q}(\tc_n))\leq \varepsilon(p)n.
\]
Moreover, $\varepsilon(p)\to 0$ as $p\to\infty$.
\end{lemma}
\begin{proof}
We begin by decomposing the toll function $f_{p,q}$ into two parts 
$
f_{p,q}(T)=\hf(T)+ f_{p,q}^{(1)}(T),
$ where $f_{p,q}^{(1)}(T)=\E(\n(\tc_{|T|}))=\mu_{|T|}$. What is important to notice in this decomposition is that $\E(\hf (\tc_k))=0$ for every $k$, and $f_{p,q}^{(1)}(T)$ depends only on the size of $T$.   

 The above decomposition of the toll function leads to a decomposition of the corresponding additive function $F_{p,q}$ as follows 
\[
F_{p,q}(T)=\hF{(T)}+F_{p,q}^{(1)}(T).
\]

Let us first deal with the functional $F_{p,q}^{(1)}(T)$ whose toll function $f_{p,q}^{(1)}(T)$ depends only on the size of $T$. We trivially have 
\[
\E f_{p,q}^{(1)}(\tc_n)=
\begin{cases}
    \mu_n \ &\text{if }\ p\le n<q\\
    0 \ &\text{otherwise }.
\end{cases}
\]
We have already established from Lemma~\ref{evaluate n(Tn)} that $\mu_n=\E(\n(\htc))+O(n^{-1/2})$. The term $\E(\n(\htc))$  is a constant which we shall denote by $\theta$. For an additive functional whose toll function depends only on the size of the tree, one can make use of \cite[Theorem 6.7]{J16}. When applied to $F_{p,q}^{(1)}$, we obtain 
\begin{equation}\label{eq:varp1}
    \frac{1}{n}\Var(F_{p,q}^{(1)}(\tc_n))=\frac{1}{n}\Var\left(F_{p,q}^{(1)}(\tc_n)-\theta n\right)\ll \left(\sup_{p \leq k<q}|\mu_k-\theta|+\sum_{p\le k<q}\frac{|\mu_k-\theta|}{k}\right)^2.
\end{equation}
By the asymptotic estimate of $\mu_k$, we have 
\begin{equation}\label{eq:varp2}
    \sup_{p \leq k<q}|\mu_k-\theta|+\sum_{p\le k<q}\frac{|\mu_k-\theta|}{k}\ll \sup_{k\geq p}k^{-1/2}+\sum_{k\geq p}k^{-3/2}.
\end{equation}

For $\Var(\hF(\tc_n))$, as in \cite[Equation (38)]{ralaivaosaona2020} which makes use of \cite[Equation (6.28)]{J16}, we obtain
\[
\frac{1}{n}\Var(\hF(\tc_n))\ll \sum_{p\le k\leq n}\frac{n^{1/2}}{(n-k+1)^{1/2}k^{3/2}}\left|\E(\hf(\tc_k)\hF(\tc_k))\right|.
\]
We have already established in Lemma~\ref{lem:toll_add} that $|\E(\hf(\tc_k)\hF(\tc_k))|=O(1)$ where the right-hand side is independent of $p,q$ and $k$. Moreover, we have  
\begin{align*}
\sum_{p\le k\leq n}\frac{n^{1/2}}{(n-k+1)^{1/2}k^{3/2}}
& = \sum_{p\leq k\leq n/2}\frac{n^{1/2}}{(n-k+1)^{1/2}k^{3/2}}+\sum_{ n/2<k\leq n}\frac{n^{1/2}}{(n-k+1)^{1/2}k^{3/2}}\\
& <  2^{1/2}\sum_{p\leq k\leq n/2}\frac{1}{k^{3/2}}+\frac{2^{3/2}}{n}\sum_{1\leq k<1+n/2}\frac{1}{k^{1/2}}\\
& \ll \sum_{k\geq p}\frac{1}{k^{3/2}}+\frac{1}{p^{1/2}},
\end{align*}
provided that $p\leq n$, otherwise, the sum would be zero. Therefore, we deduce that 
\begin{equation}\label{eq:varp3}
    \frac{1}{n}\Var(\hF(\tc_n))=O\left(\sum_{k\geq p}\frac{1}{k^{3/2}}+\frac{1}{p^{1/2}}\right),
\end{equation}
where the right side is a function of $p$ only that tends zero as $p\to\infty$. To complete the proof of the lemma, we use the Minkowski inequality, which gives
\[
\sqrt{\Var(F_{p,q}(\tc_n))}\leq \sqrt{\Var(\hF(\tc_n))}+\sqrt{\Var(F_{p,q}^{(1)}(\tc_n))}.
\]
The lemma follows easily by combining Equations~\eqref{eq:varp1}--\eqref{eq:varp3}. 
\end{proof}

Observe that if we set $p=1$ and $q=\infty$, then Lemma~\ref{lem:varesti} shows that $\Var(\N(\tc_n))=O(n)$. We prove the central limit theorem in the next section.

\section{Proof of the central limit theorem}

 The key ingredient in the proof of the central limit theorem is the use of the so-called truncation argument, which is based on a classical result in probability theory; see, for example, \cite[Theorem 4.2]{billingsley1968convergence} and \cite[Theorem 4.28]{kallenberg2002foundations}. This technique is quite effective in proving central limit theorems for parameters of random trees; it has been used several times in the past, see for example \cite{holmgren2015limit, MR3984050, ralaivaosaona2020}. Since the argument is based on $L^{2}$-convergence, we also obtain a convergence in variance as a bonus. The idea is summarised in the next lemma whose statement was taken from \cite{ralaivaosaona2020}.
\begin{lemma}\label{lem:0}
	Let $(X_n)_{n\geq 1}$ and $(W_{p,n})_{p,n\geq 1}$ be sequences of centred random variables. If for some random variables $W_p$, $p = 1, 2, \dots$, and $W$ we have
	\begin{itemize}
		\item $W_{p,n}\dto_n W_p$ for every $p\geq 1$, and  $ W_p \dto_p W,$
		\item $\Var(X_n-W_{p,n})=O(\sigma^2_p)$ uniformly in $n$, and $\sigma^2_p\to_p0$,
	\end{itemize}
	then 
	$
	X_n\dto_n W.
	$
	
	If we assume further that $X_n$, $W_{p,n}$, $W_p$, and $W$ have finite second moments, for every $n,p \geq 1$, and 
	\begin{itemize}
	\item $\Var W_{p,n} {\to}_n \Var W_p$, and  $ \Var W_p {\to}_p \Var W,$,
	\end{itemize}
	then we also have $
	\Var X_n {\to}_n \Var W.  
	$ 

\end{lemma}

We are now ready to prove the central limit theorem in our main theorem.

\begin{proof}[Proof of the CLT in Theorem~\ref{main theorem}]
We apply the above lemma with 
\begin{align*}
X_n  =n^{-1/2}(\N(\tc_n)-\E(\N(\tc_n)), \ \text{and} \  
W_{p,n}  = n^{-1/2}(F_{1,p}(\tc_n)-\E(F_{1,p}(\tc_n))).
\end{align*}
It is clear that 
\[
X_n-W_{p,n} = n^{-1/2}(F_{p,\infty}(\tc_n)-\E(F_{p,\infty}(\tc_n))). 
\]
Hence, by Lemma~\ref{lem:varesti}, $\Var(X_n-W_{p,n})\leq \varepsilon(p)$ which tends to 0 as $p\to\infty$.  Thus, the second condition in Lemma~\ref{lem:0} is satisfied. 

On the other hand, for every fixed $p$ the toll function $f_{1,p}$ associated with $F_{1,p}$ is bounded and local, so by \cite[Theorem 1.13]{J16}, there exists $\gamma^2_p\geq 0$ such that $W_{p,n}\dto_n\mathcal{N}(0,\gamma^2_p)$ and $\Var W_{p,n}\to_n \gamma^2_p$. 

To verify that the first condition in Lemma~\ref{lem:0} is satisfied, we need to show that the sequence $(\gamma^2_p)_p$ is convergent. To achieve this, notice that for $p\leq q$, the additive functional $F_{1,q}-F_{1,p}=F_{p,q}$. Therefore, by Lemma~\ref{lem:varesti}, we obtain
\[
|\gamma_q-\gamma_p|^2\leq \lim_{n\to\infty} \Var(W_{q,n}-W_{p,n})\leq \varepsilon(p).
\]
This implies that $(\gamma^2_p)_p$ is a Cauchy sequence, hence it is convergent. Thus, there exists $\gamma\geq 0$ such that $\gamma_p^2\to \gamma^2$ as $p\to\infty$ and the corresponding $W$ in this case satisfies 
$
W \sim  \mathcal{N}(0,\gamma^2).
$
Furthermore, we also deduce that $\mathrm{Var}(X_n)\to  \gamma^2$ as $n\to\infty$.
\end{proof}

\section{Non-degeneracy}
In the previous section, we proved that, under the conditions of Theorem~\ref{main theorem} on $\st$ and $\xi$, 
\[
\frac{N_{\st}(\tc_n)-\mu n}{\sqrt{n}}\dto \mathcal{N}(0,\gamma^2) \ \text{ and } \ \mathrm{Var}(N_{\st}(\tc_n))=\gamma^2n+o(n),
\]
as $n\to\infty$. These results are of limited significance if $\gamma$ vanishes. So, in this section, we look at the cases in which this can happen.  It is clear that $\gamma$ is a function of $\st$ and $\xi$, and we want to provide a set of conditions on $\st$ and $\xi$ which guarantee that $\gamma$ does not vanish. Note that it is not required in this problem that $\mathbb{P}(\tc =\st)\neq 0$. 

\begin{assumption}\label{ass:1}
        The pair $(\st,\xi)$ satisfies the conditions of Theorem~\ref{main theorem} and there exist trees $\tau_1$ and $\tau_2$ such that 
\begin{itemize}
    \item[(A1)] $\mathbb{P}(\tc=\tau_i)>0$ for $i\in \{1,2\}$,
    \item[(A2)] $|\tau_1|=|\tau_2|$,
    \item[(A3)] $\tau_1^{(M-1)}=\tau_2^{(M-1)}$ (as rooted plane trees)
    \item[(A4)] $N_{\st}(\tau_1)\neq N_{\st}(\tau_2) $,
\end{itemize}
where $M$ is the height of the tree $\st$.
\end{assumption}
    
\begin{proposition}\label{prop:gamneq0}
    Assume that the pair $(\st,\xi)$ satisfies Assumption~\ref{ass:1}, then $\gamma> 0$, i.e, 
    \[
    \frac{N_{\st}(\tc_n)-\mu n}{\gamma \sqrt{n}}\dto \mathcal{N}(0,1) \ \text{ and }\ \mathrm{Var}(N_{\st}(\tc_n))\sim \gamma^2 n,
    \]
    as $n\to\infty$.
\end{proposition}

\begin{proof}
    The proof is based on a method first applied in \cite[Subsection 6.1]{ralaivaosaona2020}. Consider the structure $\tc_n^*$ obtained from $\tc_n$ after deleting every occurrence of $\tau_1$ and $\tau_2$ as fringe subtrees, while marking the nodes to which these copies were attached. Let $n^*$ denote the number of these marked nodes. By the law of total variance 
    \[
    \mathrm{Var}(N_{\st}(\tc_n))\geq \mathbb{E}\left(\mathrm{Var}(N_{\st}(\tc_n)\ |\ \tc_n^*)\right). 
    \]

Let us estimate the right-hand side. Under Assumption~\ref{ass:1}, define
\[
\delta := N_{\st}(\tau_1) - N_{\st}(\tau_2)
\quad \text{and} \quad
p_i := \mathbb{P}(\tc = \tau_i) \quad \text{for each } i \in \{1, 2\}.
\]
Note that $\delta\neq 0$, and $p_1, p_2>0$. To construct \( \tc_n \) from \( \tc_n^* \), we reattach either \( \tau_1 \) or \( \tau_2 \) to each marked node of \( \tc_n^* \). At each marked node, the probability of attaching \( \tau_i \) is given by
\[
\frac{p_i}{p_1 + p_2}, \quad \text{for } i \in \{1, 2\}.
\]
Furthermore, the choices of attaching \( \tau_1 \) or \( \tau_2 \) at different marked nodes are independent. Observe also that, given $\tc_n^*$, the only variation in $N_{\st}(\tc_n)$ comes from the occurrences of $\st$ in those copies of $\tau_i$; this is because of Statement (A3) in Assumption~\ref{ass:1}. Thus, conditioning on $\tc_n^*$, the variable $N_{\st}(\tc_n)$ can be decomposed as follows:
    \[
    N_{\st}(\tc_n)=h(\tc_n^*)+\delta \cdot Y_n,
    \]
    where the variable $h(\tc_n^*)$ is $\tc_n^*$--measurable, and the conditional distribution of $Y_n$ given $\tc_n^*$ is the binomial $ B(n^*,p)$, where $p=p_1/(p_1+p_2)$. Hence, we deduce that 
    \[
    \mathrm{Var}(N_{\st}(\tc_n)\ |\ \tc_n^*)=\frac{\delta^2 p_1p_2}{(p_1+p_2)^2} n^*.
    \]
    
    On the other hand, $n^*$ is the number of occurrences of $\tau_1$ and $\tau_2$ as a fringe subtree in $\tc_n$. Hence,  we have $\mathbb{E}(n^*)\sim (p_1+p_2)n$, see for example \cite[Theorem 1.3]{J16}. Therefore, 
    \[
    \mathbb{E}\left(\mathrm{Var}(N_{\st}(\tc_n)\ |\ \tc_n^*)\right)\sim \frac{\delta^2 p_1p_2}{p_1+p_2} n.
    \]
    Finally, we obtain $\mathrm{Var}(N_{\st}(\tc_n))\gg n$, which proves that $\gamma> 0$.
\end{proof}

\begin{proposition}\label{prop:lin}
    Assume that the pair $(\st,\xi)$ satisfies the conditions of Theorem~\ref{main theorem} and such that $\gamma= 0$, then there exist a function $g:\mathbb{N}\to \mathbb{R}$ and constants $\alpha_1, \alpha_2,\ \ldots$, such that 
    \[
    N_{\st}(T)=g(|T|)+\sum_{i}\alpha_i n_{\st_i}(T)
    \]
    for all $T$ with $\mathbb{P}(\tc=T)>0$, where $\st_1,\ \st_2,\ \ldots$ are the fringe subtrees of $\st$.
\end{proposition}
    
\begin{proof}
    We choose a fixed tree $\tau$ with height $M-1$, such that $\mathbb{P}(\tc=\tau)>0$. For each tree $T$ with $\mathbb{P}(\tc=T)>0$, let $\tau(T)$ be the tree obtained from attaching $T$ to the left-most leaf of $\tau$ at level $M-1$. Since, we are assuming that $\gamma=0$, by Proposition~\ref{prop:gamneq0}, $N_{\st}(\tau(T))$ depends only on the size of $T$ otherwise it would be possible to construct the trees $\tau_1$ and $\tau_2$ in Assumption~\ref{ass:1}, and contradicting the assumption that $\gamma=0$.
    
    On the other hand, by construction $N_{\st}(\tau(T))$ equals $N_{\st}(T)$ plus a linear combination of $n_{\st_i}(T)$'s where the coefficients depends on $\st$ and $\tau$ only. 
\end{proof}

Note that, in Proposition~\ref{prop:lin}, most of the coefficients $\alpha_i$ are in fact zero. For example, given a fringe subtree $\st_i$ of $\st$ with height at least $1$, the coefficient $\alpha_i$ vanishes unless the root of $\st_i$ is the only node at that level in $\st$ that is not a leaf. 

\begin{corollary}
    Assume that the pair $(\st,\xi)$ satisfies the conditions of Theorem~\ref{main theorem} and such that $\gamma= 0$, then
    \[
    \mathrm{Var}(N_{\st}(\tc_n))=O(1) \quad \text{ as } \quad n\to \infty.
    \]
\end{corollary}

\begin{proof}
    This follows easily from Proposition~\ref{prop:lin} and Part (c) of Lemma~\ref{expectation of products of n(T) with finite tree size}.
\end{proof}

There are nontrivial degenerate cases. For example, taking $\st$ to be a path of fixed length already yields degeneracies. Additional examples arise if the $\gcd$-condition on $\xi$ is relaxed; see \cite{J21} for a detailed discussion. Further special cases also occur if one counts occurrences of subtrees from a finite collection of trees, rather than a single tree~$\st$. We do not pursue this direction in this work.

\section{The moment assumption in Theorem~\ref{main theorem}}\label{sec5}

 The central limit theorem stated in Theorem~\ref{main theorem} is guaranteed by the assumption that $\E(\xi^{2\Delta+1})$ is finite. This assumption is natural, as such a moment condition is necessary to make sure that $\E(\n(\tc_n)\, n_{\st'}(\tc_n))$ remains bounded for any two trees $\st$ and $\st'$, each having maximum degree at most $\Delta$, as shown in Part~(c) of Lemma~\ref{lem:0}. However, it is not clear whether Theorem~\ref{main theorem} would still hold if this condition were not satisfied. While it is difficult to answer this completely, we provide two examples which indicate that this moment condition is not too far from the optimal condition that guarantees the full statement of Theorem~\ref{main theorem}.
  
\begin{example}[$\st$ is a path]
   Let us consider the case where $\st$ is a path of length $2$, i.e., we can take $\Delta=1$. Then, as shown in Janson~\cite{J21}, 
    \[
    \N(\tc_n)=n-w_1(\tc_n)-1\quad \text{ and } \quad \mathrm{Var} (\N(\tc_n))=\mathrm{Var}(w_1(\tc_n)).
    \]
    Now, if we assume that $\E(\xi^3)=\infty$ but $\E(\xi^{3-\epsilon})<\infty$ for some $\epsilon \in (0, \tfrac{1}{2})$, then $\E (w_1(\tc_n))\to \E(w_1(\htc))=\E(\xi^2)<\infty$. However, using \eqref{eq:ntohat}, we have
    \[
    \E (w_1(\tc_n)^2)\gg \sum_{k\leq \sqrt[3]{n}}k^2\Prob(w_1(\htc)=k)=\sum_{k\leq \sqrt[3]{n}}k^3\Prob(\xi=k) \to\infty.
    \]
    Similarly, using \eqref{eq:ntohat} again, we obtain the upper bound
    \[
    \E (w_1(\tc_n)^2)\ll \sqrt{n}\sum_{k\leq \sqrt{n}} k^2\Prob(\xi=k)+  n^{(2\epsilon+1)/2}\sum_{\sqrt{n}<k< n/2} k^{3-\epsilon}\Prob(\xi=k)+n^2\prob{w_1(\tc_n)\geq n/2}.
    \]
    The first two terms on the right-hand side are both $o(n)$. For the third term, observe that
    \[
    \prob{w_1(\tc_n)\geq n/2} \leq \frac{\prob{w_1(\tc)\geq n/2}}{\prob{|\tc|=n}} \ll n^{3/2}\prob{w_1(\tc)\geq n/2}, 
    \]
    and by Markov's inequality, we have 
    \[
    \prob{w_1(\tc)\geq n/2} = \prob{w_1(\tc)^{3-\epsilon}\geq (n/2)^{3-\epsilon}}\ll n^{-3+\epsilon}.
    \]
    Hence, we also have 
    \[
    n^2\prob{w_1(\tc_n)\geq n/2} \ll n^{(2\epsilon+1)/2}=o(n).
    \]
    Therefore, we deduce that the variance $\mathrm{Var}(\N(\tc_n))$ is not linear in $n$, but it tends to infinity. Together with the fact that $w_1(\tc_n)\geq 0$ and $\E(w_1(\tc_n)) = O(1)$, this implies that the fluctuations of $\N(\tc_n)$ cannot be Gaussian.

\end{example}  

\begin{example}[$\st$ is a star]
    Here, we let $\st$ be a star with $\Delta$ leaves where $\Delta\geq 3$. Now assume that $\E(\xi^{\Delta+2})<\infty$ and $\E(\xi^{2\Delta})=\infty$. As in the proof of Lemma~\ref{lem:varesti}, we split the toll function as
    $
    \n=\tilde{f}+f^{(1)},
    $
    where $f^{(1)}(T)=\mu_{|T|}$ and $\tilde{f}(T)=\n(T)-\mu_{|T|}$. Let $\tilde {F}$ and $F^{(1)}$ be the additive functionals associated with $\tilde{f}$ and $f^{(1)}$, respectively. From the Minkowski inequality, we have
    \begin{equation}\label{eq:var_break}
        \sqrt{\Var(\N(\tc_n))}\geq \sqrt{\Var(\tilde{F}(\tc_n))}-\sqrt{\Var(F^{(1)}(\tc_n))}.
    \end{equation}
    Using the same argument as in the derivation of Equation~\eqref{eq:varp1}, we obtain from \cite[Theorem 6.7]{J16} that 
    \begin{equation}\label{eq:varF1}
        \sqrt{\Var(F^{(1)}(\tc_n))}\ll \sqrt{n} \left(\sup_{k}|\mu_k-\theta|+\sum_{k=1}^{\infty}\frac{|\mu_k-\theta|}{k}\right)=O(\sqrt{n}),
    \end{equation}
    where the last estimate follows from Lemma~\ref{evaluate n(Tn)}. Next we estimate $\Var(\tilde{F}(\tc_n))$. From \cite[Equation~(6.27)]{J16}, we infer
    \[
    \frac{1}{n}\mathrm{Var}(\tilde  F(\tc_n))=\sum_{k=1}^{n}
\frac{\mathbb{P}(S_{n-k}=n-k)}{\mathbb{P}(S_n = n-1)}
\, \pi_k \, \mathbb{E}\!\left(\tilde  f(\tc_k)\bigl(2\tilde  F(\tc_k)-\tilde f(\tc_k)\bigr) \right).
    \]
    where $S_m$ is the sum of $m$ independent copies of $\xi$, and $\pi_k=\prob{|\tc|=k}$. We write
    \[
    \tilde F(\tc_k)=\sum_{\ell\geq 0}\sum_{d(v)=\ell} \tilde f(\tc_{k,v}),
    \]
    and obtain 
    \[
    \tilde f(\tc_k)\bigl(2\tilde F(\tc_k)-\tilde f(\tc_k)\bigr)=\tilde f(\tc_k)^2+2\sum_{\ell\geq 1}\sum_{d(v)=\ell} \tilde f(\tc_k)\tilde f(\tc_{k,v}).
    \]
    Since $\st$ is a star, the expectation of the second term on the right-hand side is zero (using the same argument as in the proof of Lemma~\ref{lem:toll_add}). Hence,
    \[
    \mathbb{E}\!\left( \tilde  f(\tc_k)\bigl(2\tilde F(\tc_k)-\tilde f(\tc_k)\bigr) \right)=\E(f(\tc_k)^2)=\E(\n(\tc_k)^2)-\mu_k^2.
    \]
    Since $\E(\xi^{\Delta+2})<\infty$, $\mu_k=O(1)$. Therefore, using \eqref{eq:ntohat}, there exists a constant $c_1>0$ such that 
    \[
    \mathbb{E}\!\left( \tilde  f(\tc_k)\bigl(2\tilde  F(\tc_k)-\tilde  f(\tc_k)\bigr) \right) \gg \E(w_1(\tc_k)^{2\Delta})\gg \sum_{j\leq c_1\sqrt{k}}j^{2\Delta+1}\prob{\xi=j}. 
    \]
    From the above, together with \cite[Lemma~5.2, Part~(i)]{J16}, we deduce that there exist constants $c_2,c_3>0$ such that
    \begin{align*}
        \frac{1}{n}\mathrm{Var}(\tilde  F(\tc_n)) 
         \gg \sum_{k\leq c_2n}\frac{1}{k^{3/2}}\sum_{j\leq c_1\sqrt{k}}j^{2\Delta+1}\prob{\xi=j}
         \gg  \sum_{j\leq c_3\sqrt{n}} j^{2\Delta}\prob{\xi=j}.
    \end{align*}
    The latter clearly tends to $\infty$ as $n\to\infty$ because of the assumption $\E(\xi^{2\Delta})=\infty$. Therefore, combining this with Equations \eqref{eq:var_break} and \eqref{eq:varF1}, the variance $\Var(\N(\tc_n))$ in this case grows superlinearly in $n$, that is 
    \[
    \frac{1}{n}\Var(\N(\tc_n))\to \infty \quad \text{ as } \quad n\to \infty. 
    \]
\end{example}
\bibliographystyle{plain}
\bibliography{subtree}
\end{document}